\documentclass[proceedings%
]{aofa}

\usepackage[english]{babel} 
\usepackage[T1]{fontenc}
\usepackage[utf8]{inputenc}

\usepackage{lmodern} 

\usepackage{graphicx} 
\usepackage{float}
\usepackage{multirow}
\usepackage{tikz} 
\usepackage{footnote}   
\usetikzlibrary{calc}
\usetikzlibrary{arrows}
\usetikzlibrary{decorations.pathreplacing}  
\usepgflibrary{shapes.geometric}

\usepackage{ifthen}
\usepackage{xparse}  

\usepackage{amsmath}
\usepackage{amsfonts}
\usepackage{amssymb,latexsym}

\usepackage{setspace}

\usepackage{a4wide} 
\usepackage{makeidx} 
\usepackage{etoolbox} 
\usepackage{paracol} 
\usepackage{array}   

\usepackage{color}
\usepackage{url}
\usepackage{paralist}
\usepackage{cite}

\usepackage[math]{cellspace}
\newcommand{\stepset}{\mathcal{S}}

\newcommand{\walksym}{\omega}
\newcommand{\walk}[1]{\walksym_{#1}}

\newcommand{\neighrayrho}{\Omega \setminus (\rho,\infty)}

\newcommand{\E}{\mathbb{E}}

\newcommand{\N}{\mathbb{N}}
\newcommand{\PR}{\mathbb{P}}

\newcommand{\R}{\mathbb{R}}

\newcommand{\V}{\mathbb{V}}
\newcommand{\Z}{\mathbb{Z}}

\newcommand{\Bc}{\mathcal{B}}

\newcommand{\Ec}{\mathcal{E}}

\newcommand{\Hc}{\mathcal{H}}
\newcommand{\Nc}{\mathcal{N}}
\newcommand{\Rc}{\mathcal{R}}

\newcommand{\Wc}{\mathcal{W}}

\newcommand{\LandauO}{\mathcal{O}}
\newcommand{\Landauo}{o}

\newcommand{\Mc}{\mathcal{M}}

\newcommand{\proofend}
{ \hfill $_{\square}$}

\newcommand{\diamondend}
{ \hfill ${\lozenge} $}

\newtheorem{theo}{Theorem}[section]
\newtheorem{lemma}[theo]{Lemma}
\newtheorem{prop}[theo]{Proposition}

\newenvironment{definition}[1][]{\refstepcounter{theo} \medskip \noindent \textbf{\textit{Definition \thetheo~(#1)}} }{ \diamondend }

\newenvironment{remark}[1][]{\refstepcounter{theo}\medskip \noindent\textbf{\textit{Remark \thetheo~(#1)}} }{ \hfill $_{\blacksquare} $\\ }

\renewenvironment{proof}[1][]
{ \noindent
	\ifthenelse{\equal{#1}{}}
		{\textbf{Proof:}}
		{\textbf{Proof #1\ignorespaces:}}
}
{ \proofend\medskip }

\newenvironment{hypo}[1][]{\medskip \noindent %
\textbf{\textit{Hypothesis #1:}} }
{ \diamondend \medskip }

\newcommand{\roleW}{r\^ole}
\newcommand{\role}{\roleW~}

\begin{document}


\title{
A half-normal distribution scheme for generating functions and the unexpected behavior of Motzkin paths
}
\author{Michael Wallner\thanks{Institute of Discrete Mathematics and Geometry, TU Wien, Wiedner Hauptstr. 8-10/104, A-1040 Wien, Austria}}

\author[M. Wallner]{Michael Wallner\addressmark{1} }

\title[A half-normal distribution scheme]{A half-normal distribution scheme for generating functions and the unexpected behavior of\\Motzkin paths}

\address{\addressmark{1}Institute of Discrete Mathematics and Geometry, TU Wien, Wiedner Hauptstr. 8-10/104, A-1040 Wien, Austria}


\maketitle

\begin{abstract}
We present an extension of a theorem by Michael Drmota and Mich\`ele Soria [Images and Preimages in Random Mappings, $1997$] that can be used to identify the limiting distribution for a class of combinatorial schemata. This is achieved by determining analytical and algebraic properties of the associated bivariate generating function. We give sufficient conditions implying a half-normal limiting distribution, extending the known conditions leading to either a Rayleigh, a Gaussian, or a convolution of the last two distributions. We conclude with three natural appearances of such a limiting distribution in the domain of Motzkin paths. 
\end{abstract}

\section{Introduction}
\label{sec:intro}

Generating functions have proved very useful in the analysis of combinatorial questions. The approach builds on general principles of the correspondence between combinatorial constructions and functional operations. The symbolic method~\cite{flaj09} provides a direct translation of the structural description of a class into an equation on generating functions. In~\cite{DrSo95}, Drmota and Soria provided general methods for the analysis of bivariate generating functions $F(z,u) = \sum f_{nk} z^n u^k$. In general,~$n$ is the length or size, and $k$ is the value of a ``marked'' parameter. 

They continued their work in~\cite{DrSo97}, wherein they derived three general theorems which identify the limiting distribution for a class of combinatorial schemata from certain properties of their associated bivariate generating function. These lead to a Rayleigh, a Gaussian, or a convolution of both distributions. 
Especially for a Gaussian limit distribution there are many schemata known: Hwang's quasi-powers theorem~\cite{hwan98}~or~\cite[Theorem~IX.8]{flaj09}, the supercritical composition scheme~\cite[Proposition~IX.6]{flaj09}, the algebraic singularity scheme \cite[Theorem~IX.12]{flaj09}, an implicit function scheme for algebraic singularities \cite[Theorem~2.23]{drmo09}, or the limit law version of the Drmota-Lalley-Woods theorem~\cite[Theorem~8]{badr15}. But such schemata also exist for other distributions, like e.g., the Airy distribution, see~\cite{bfss01}. In general it was shown in~\cite{bbpb12} and \cite[Theorem~10]{badr15} that even in simple examples ``any limit law'', in the sense that the limit curve can be arbitrarily close to any c\`adl\`ag multi-valued curve in $[0,1]^2$, is possible. 

In this paper we extend the work of~\cite{DrSo97}, by providing an additional limit theorem, Theorem~\ref{theo:theo4}, which reveals a half-normal distribution. This distribution is generated by the absolute value $|X|$ of a normally distributed random variable $X$ with mean $0$. We will encounter several distributions, whose most important properties are summarized in Table~\ref{tab:compprobdis}. 

We also present three natural appearances of this distribution in combinatorial constructions. In particular we consider Motzkin walks. Despite them being well-studied objects~\cite{motz48,dosh77,bern99}, they still hide some mysterious properties. Our applications extend some examples of random walks presented by Feller in \cite[Chapter III]{fell68} to Motzkin walks. We show that the same phenomena appear which, to quote Feller, ``not only are unexpected but actually come as a shock to intuition and common sense''. 

\begin{table}
	\renewcommand{\arraystretch}{1.5}
	\begin{tabular}{|c||c|c|c|c|}
		\hline
		&
		\underline{\bf Geometric}
		&
		\underline{\bf Normal}
		&
		\underline{\bf Half-normal}
		&
		\underline{\bf Rayleigh}
		\\
		&
		$
			\operatorname{Geom}(p)
		$
		&
		$
			\Nc(\mu,\sigma)
		$
		&
		$
			\Hc(\sigma)
		$
		&
		$
			\Rc(\sigma)
		$
		\\
		\hline\hline Graph &
		{\includegraphics[width=0.18\textwidth]{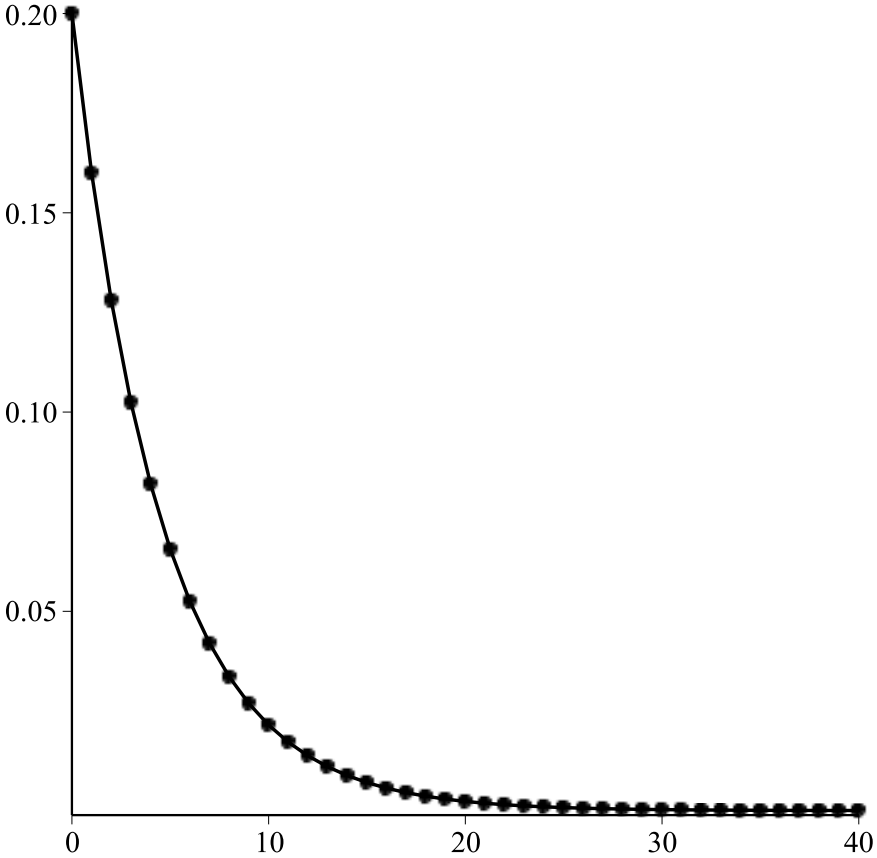}}
		&
		{\includegraphics[width=0.18\textwidth]{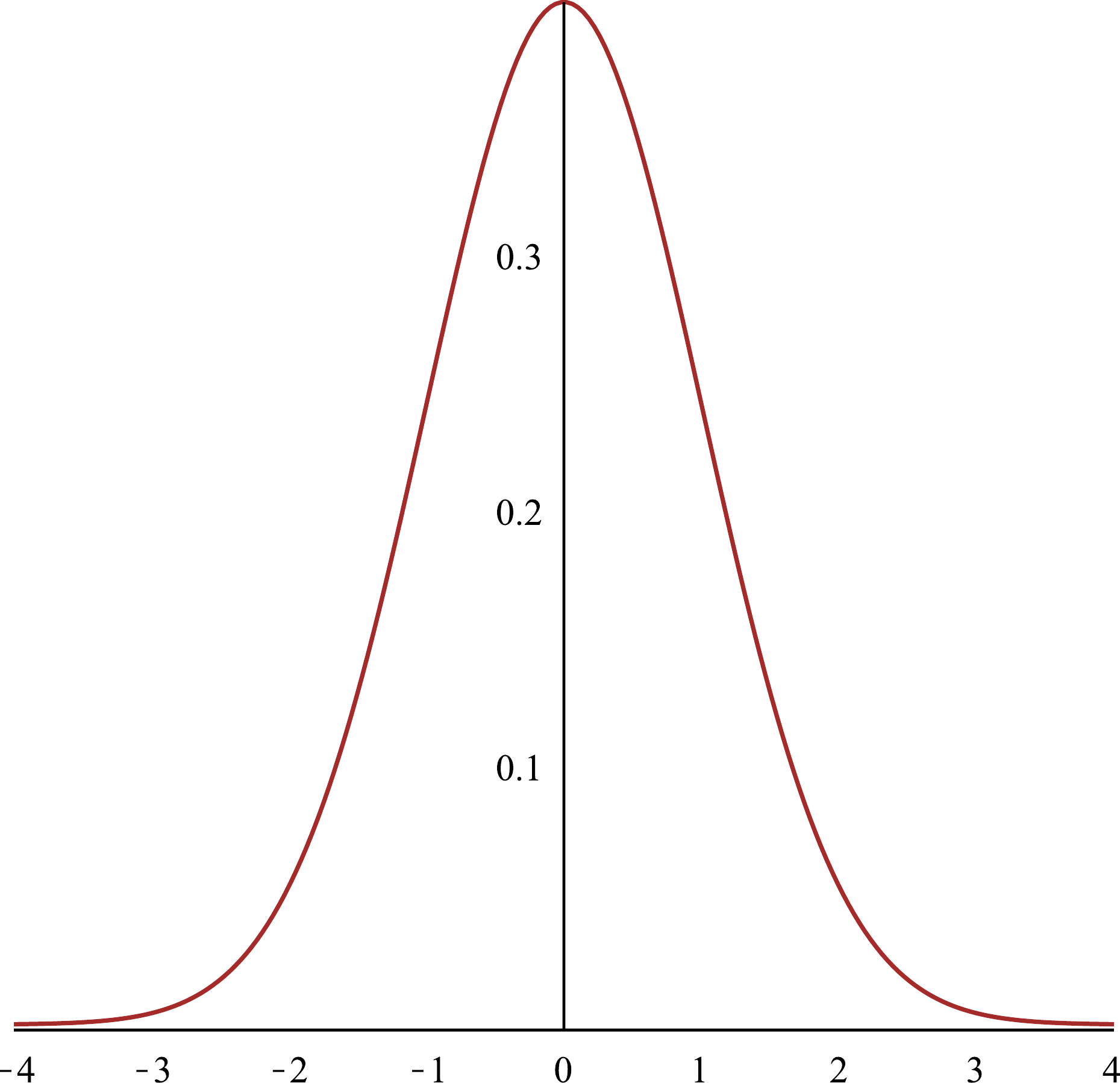}}
		&
		{\includegraphics[width=0.18\textwidth]{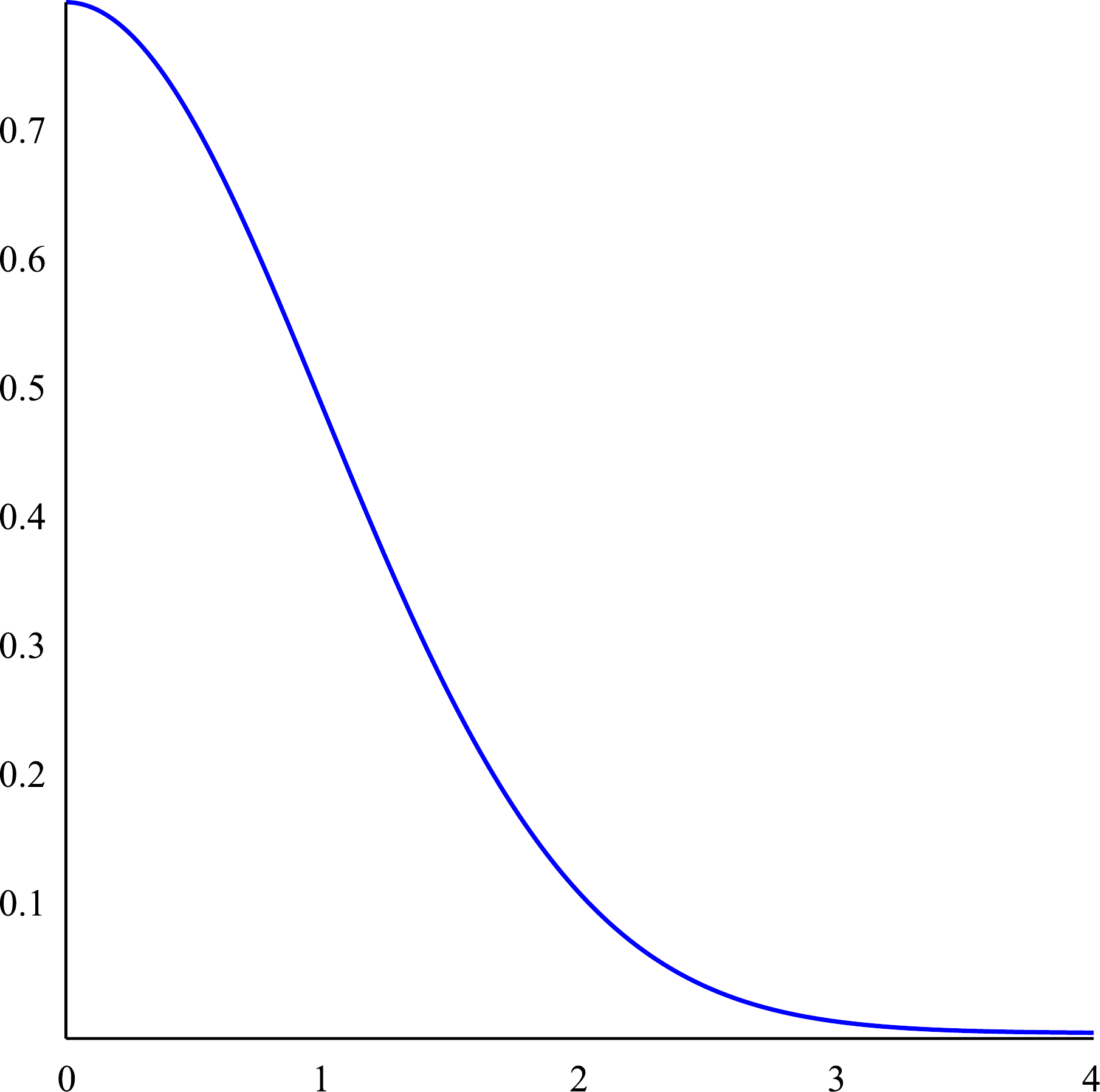}}
		&
		{\includegraphics[width=0.18\textwidth]{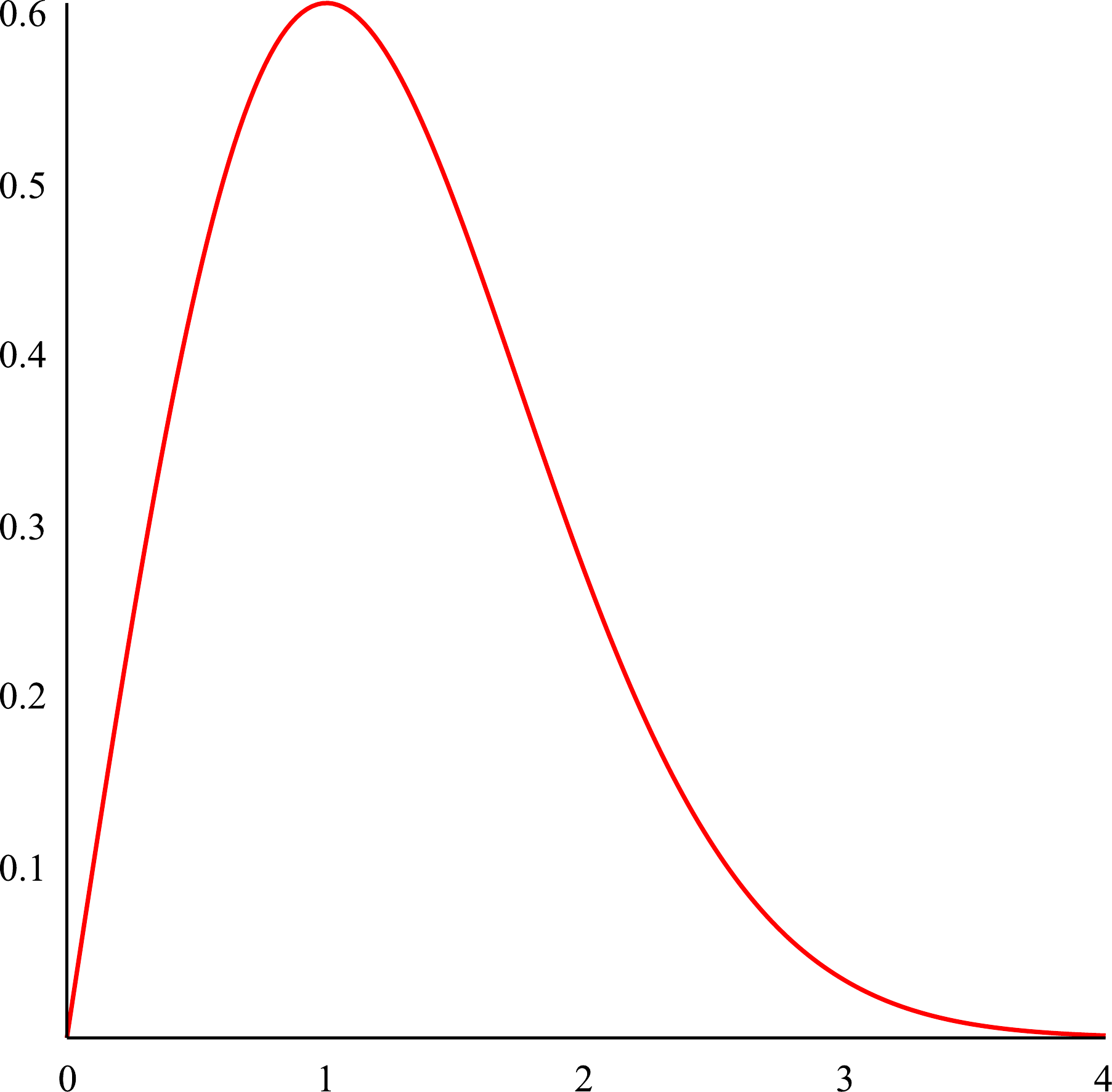}}
		\\
		\hline
		Support
		&
		$
			x \in \{0,1,\ldots\}
		$
		&
		$
			x \in \R
		$
		&
		$
			x \in \R_{\geq 0}
		$
		&
		$
			x \in \R_{\geq 0}
		$
		\\
		\hline
		PDF
		&
		$
			(1-p)^k p
		$
		&
		$
			\frac{1}{\sqrt{2 \pi \sigma^2}} \exp\left(-\frac{(x-\mu)^2}{2\sigma^2}\right)
		$
		&
		$
			\sqrt{\frac{2}{\pi \sigma^2}} \exp\left(-\frac{x^2}{2\sigma^2}\right)
		$
		&
		$
			\frac{x}{\sigma^2} \exp\left(-\frac{x^2}{2\sigma^2}\right)
		$
		\\
		\hline
		Mean
		&
		$
			\frac{1-p}{p}
		$
		&
		$
			\mu
		$
		&
		$
			\sigma\sqrt{\frac{2}{\pi}}
		$%
		&
		$
			\sigma\sqrt{\frac{\pi}{2}}
		$%
		\\
		\hline
		Variance
		&
		$
			\frac{1-p}{p^2}
		$
		&
		$
			\sigma^2
		$
		&
		$
			\sigma^2 \left(1 - \frac{2}{\pi}\right)
		$
		&
		$
			\sigma^2 \left(2 - \frac{\pi}{2} \right)
		$
		\\
		\hline
	\end{tabular}
	\renewcommand{\arraystretch}{1}
	\caption{A comparision of the geometric, normal, half-normal, and Rayleigh distribution. We will encounter all four of them in the context of Motzkin walks.}
	\label{tab:compprobdis}
\end{table}

{\bf Plan of this article.} 
First, in Section~\ref{sec:halfnormal}, we present our main contribution: a scheme for bivariate generating functions leading to a half-normal distribution.
In Section~\ref{sec:motzkin}, we introduce Motzkin paths and establish the analytic framework which will be used in the subsequent sections. 
In Section~\ref{sec:motzprop}, we apply our result to three properties of Motzkin walks: the number of sign changes, the number of returns to zero, and the height. In the case of zero drift a half-normal distribution appears. 
In Section~\ref{sec:conclusion}, we give a summary of our results.

\section{The half-normal theorem}
\label{sec:halfnormal}

Let $c(z) = \sum_n c_n z^n$ be the generating function of a combinatorial structure and $c(z,u) = \sum c_{nk} z^n u^k$ be the bivariate generating function where a parameter of interest has been marked, i.e.,~$c(z,1) = c(z)$. We introduce a sequence of random variables $X_n, n \geq 1$, defined by
\begin{align*}
	\PR[X_n = k] &= \frac{c_{nk}}{c_n} = \frac{[z^n u^k] c(z,u)}{[z^n] c(z,1)},
\end{align*}
where $\PR$ denotes the probability. As we are interested in the asymptotic distribution of the marked parameter among objects of size $n$ where $n$ tends to infinity, the probabilistic point of view is given by finding the limiting distribution of $X_n$.

Important combinatorial constructions are ``sequences'' or ``sets of cycles'' (in the case of exponential generating functions) which imply the following decomposition 
\begin{align*}
	c(z,u) &= \frac{1}{1- a(z,u)},
\end{align*}
with a generating function $a(z,u)$ corresponding to the elements of the sequence, or the cycles, respectively.
Another important and recurring phenomenon is the one of an algebraic singularity~$\rho(u)$ of the square-root type such that $a(\rho(1),1)=1$. According to further analytic properties of $a(z,u)$ the limiting distribution of $X_n$ is shown to be either Gaussian, Rayleigh, the convolution of Gaussian and Rayleigh (see \cite[Theorems 1-3]{DrSo97}), or half-normal (see Theorem~\ref{theo:theo4}). 

We start with the general form of the analytic scheme. 
In contrast to the original hypothesis~[H] in \cite{DrSo97} we call our hypothesis~[H'] because we drop the condition that $h(\rho,1) >0$ and we require it only for $\rho(u) = const$.

\begin{hypo}[{[H']}]
	Let $c(z,u) = \sum_{n,k} c_{nk} z^n u^k$ be a power series in two variables with non-negative coefficients $c_{nk} \geq 0$ such that $c(z,1)$ has a radius of convergence of $\rho > 0$. 
	
	We suppose that $1/c(z,u)$ has the local representation
	\begin{align}
		\label{eq:decompFinv}
		\frac{1}{c(z,u)} &= g(z,u) + h(z,u) \sqrt{1 - \frac{z}{\rho}},
	\end{align}
	for $|u-1| < \varepsilon$ and $|z - \rho| < \varepsilon$, $\arg(z-\rho) \neq 0$, where $\varepsilon > 0$ is some fixed real number, and $g(z,u)$, and $h(z,u)$ are analytic functions.
	Furthermore, these functions satisfy $g(\rho,1)=0$.
	
	In addition, $z=\rho$ is the only singularity on the circle of convergence $|z| = |\rho|$, and $1/c(z,u)$, respectively $c(z,u)$, can be analytically continued to a region $|z| < \rho + \delta, |u|<1+\delta, |u-1| > \frac{\varepsilon}{2}$ for some $\delta > 0$. 
\end{hypo}

\begin{theo}[Half-normal limit theorem]
	\label{theo:theo4}
	Let $c(z,u)$ be a bivariate generating function satisfying [H']. If $g_z(\rho,1) \neq 0$, $h_u(\rho,1) \neq 0$, and $h(\rho,1) = g_u(\rho,1) = g_{uu}(\rho,1) = 0$, then the sequence of random variables $X_n$ defined by
	\begin{align*}
		\PR[X_n = k] &= \frac{[z^n u^k] c(z,u)}{[z^n] c(z,1)},
	\end{align*}
	has a half-normal limiting distribution, i.e.,
	\begin{align*}
		\frac{X_n}{\sqrt{n}} \stackrel{\text{d}}{\to} \Hc(\sigma),
	\end{align*}
	where $\sigma = \sqrt{2}\frac{h_u(\rho,1)}{\rho g_z(\rho,1)}$, and $\Hc(\sigma)$ has density $\frac{\sqrt{2}}{\sqrt{\pi \sigma^2 }} \exp\left( - \frac{z^2}{2 \sigma^2} \right)$ for $z \geq 0$. Expected value and variance are given by
	\begin{align*}
		\E[X_n] &= \sigma \sqrt{\frac{2}{\pi}} \sqrt{n} + \LandauO(1) & \text{and} &&
		\V[X_n] &= \sigma^2 \left( 1 - \frac{2}{\pi} \right) n + \LandauO(\sqrt{n}).
	\end{align*}
	
	Moreover, we have the local law
	\begin{align*}
		\PR[X_n = k] &= \frac{1}{\sigma} \sqrt{\frac{2}{\pi n}} \exp\left( -\frac{k^2/n}{2\sigma^2 } \right) + \LandauO\left(kn^{-3/2}\right) + \LandauO\left( n^{-1} \right),
	\end{align*}
	uniformly for all $k \geq 0$.
\end{theo}

\begin{remark}[Non-trivial dependency of $\rho$ on $u$]
	\label{rem:strongermaintheo}
	The assumption of a constant singularity in $z$ given by $\rho$ can be weakened to a singularity $\rho(u) = \rho(1) + \LandauO((u-1)^3)$, i.e.,~$\rho'(1)=\rho''(1)=0$. However, no example is known where $\rho(u)$ is not constant in a neighborhood of $u \sim 1$.
\end{remark}

\begin{proof}[(Sketch)]
	The proof ideas are similar to the ones of \cite[{\sc Theorem 1}]{DrSo97}.
	For details on the half-normal distribution we refer to \cite{tbha14}, but all we need is the characteristic function. 
	The main idea is to derive the asymptotic form of the characteristic function of $X_n/\sqrt{n}$. Since
	\begin{align*}
		\E[ e^{i t X_n /\sqrt{n}}] &= \frac{[z^n] c(z,e^{\frac{it}{\sqrt{n}}})}{[z^n] c(z,1)},
	\end{align*}
	we need to expand $[z^n]c(z,u)$ for $u = e^{it/\sqrt{n}} = 1 + \frac{it}{\sqrt{n}} + \LandauO(n^{-1})$. To achieve this, we will apply Cauchy's integral formula for the following path of integration $\Gamma = \Gamma_1 \cup \Gamma_2$:
	\begin{align*}
		\Gamma_1 &= \left\{ z = \rho\left(1+\frac{s}{n}\right) : s \in \gamma' \right\},\\
		\Gamma_2 &= \left\{ z = R e^{i \vartheta} : R = \rho\left| 1 + \frac{\log^2 n + i}{n} \right|,~\arg\left(1 + \frac{\log^2 n + i}{n} \right) \leq |\vartheta| \leq \pi \right\},
	\end{align*}
	where $\gamma' = \{ s : |s| = 1,~\Re s \leq 0\} \cup \{ s : 0 < \Re s < \log^2 n,~ \Im s = \pm 1 \}$ is the major part of a Hankel contour $\gamma$, see Figure~\ref{fig:hankel}. 
	
	\begin{figure}[ht]
		\begin{center}	
			\includegraphics[width=0.27\textwidth]{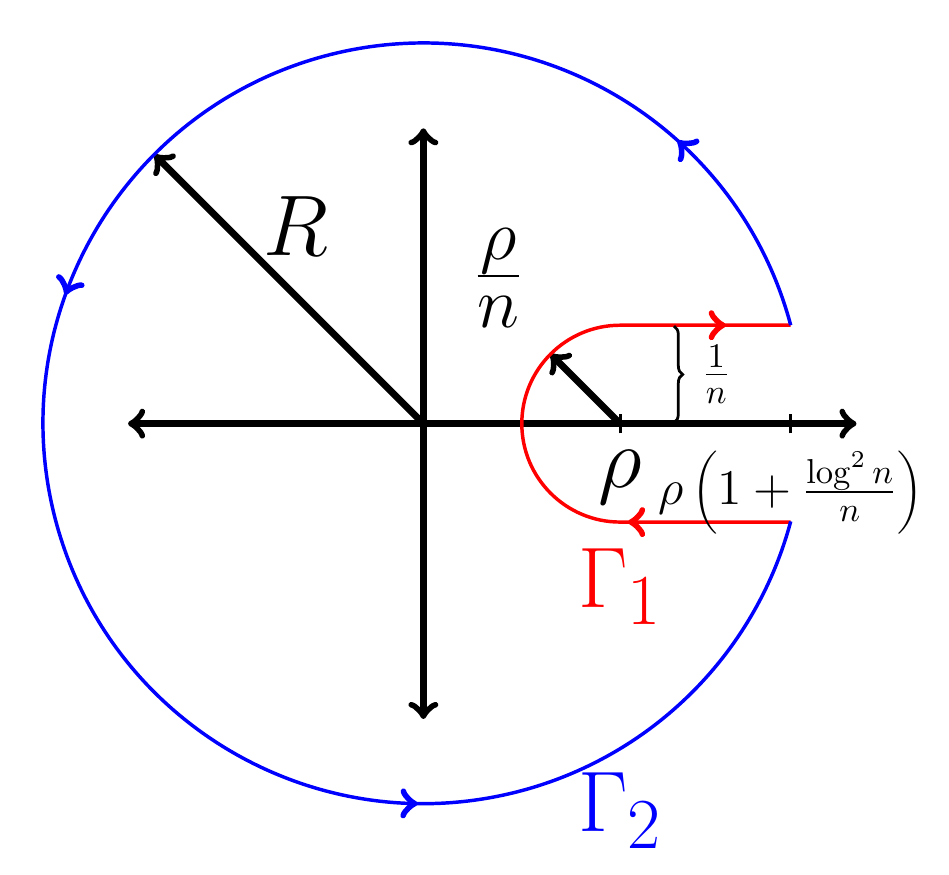} \qquad \qquad \qquad
			\includegraphics[width=0.27\textwidth]{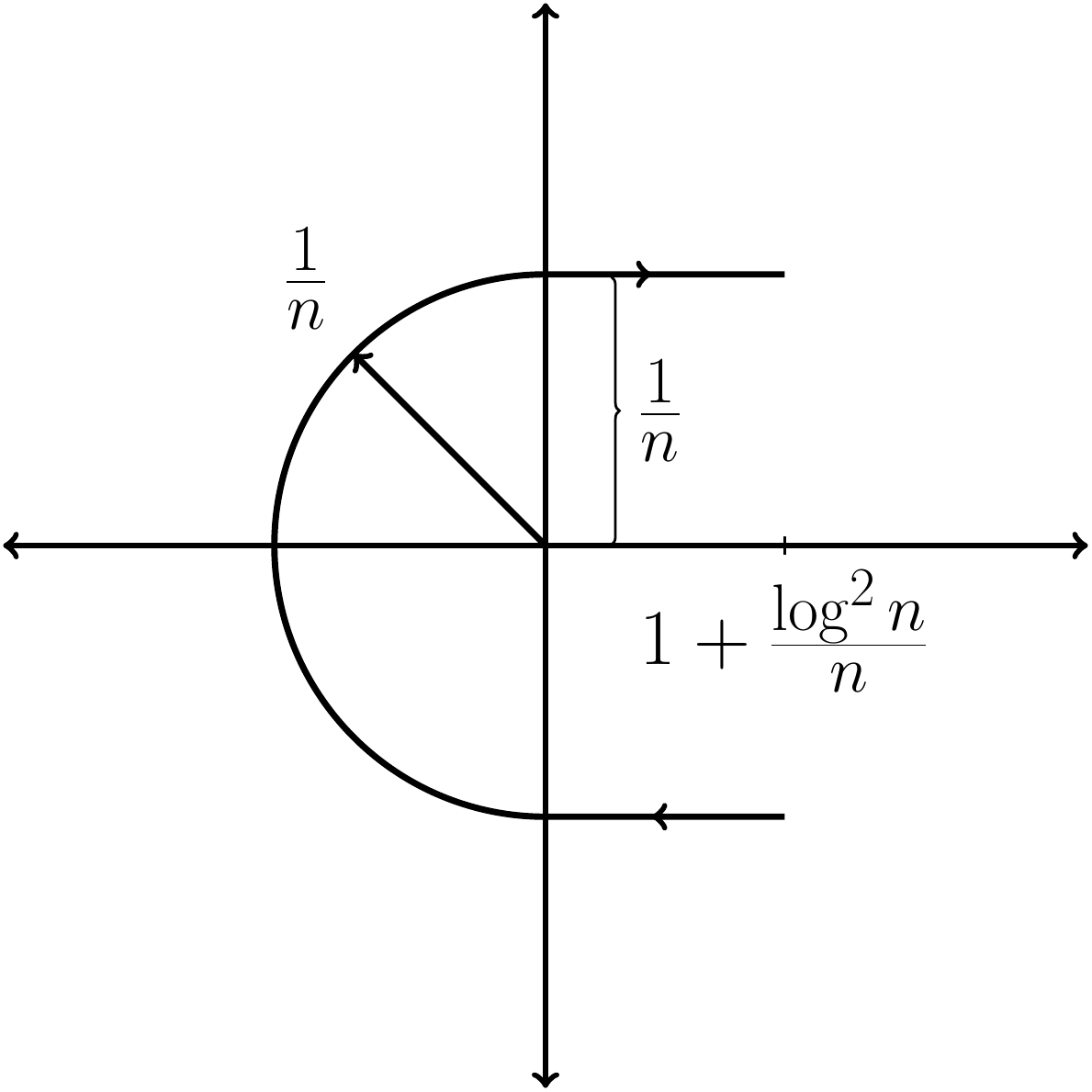}
			\caption{Hankel contour decomposition (left), and contour of $\gamma'$ (right).}
			\label{fig:hankel}
		\end{center}
	\end{figure}
	
	What remains is to investigate the parts separately: The first part gives the claimed result, whereas the second one is asymptotically negligible.
	Note that the changes in the hypothesis~[H] are responsible for the appearance of characteristic function of the half-normal distribution in the limit.
	We omit these technical steps.
\end{proof}

\section{Motzkin paths}
\label{sec:motzkin}

In this section we present needed, known results on directed lattice paths. Readers familiar with the exposition of Banderier and Flajolet \cite{bafl02} or related results may skip this section.

 \begin{definition}[Lattice paths] \label{def:LP}
 A {\it step set} $\stepset \subset \Z^2$ is a fixed, finite set of vectors $\{ (a_1,b_1), \ldots,$ $(a_m,b_m)\}$. 
An $n$-step \emph{lattice path} or \emph{walk} is a sequence $(v_1,\ldots,v_n)$ of vectors, such that $v_j$ is in $\stepset$. 
Geometrically, it is a set of points $\{\walk{0},\walk{1},\ldots,\walk{n}\}$ where $\walk{i} \in \Z^2, \walk{0} = (0,0)$ 
and $\walk{i}-\walk{i-1} = v_i$ for $i=1,\ldots,n$.
 The elements of $\stepset$ are called \emph{steps} or \emph{jumps}. 
The \emph{length} $|\walksym|$ of a lattice path is its number $n$ of jumps. 
 \end{definition}

We restrict our attention to \emph{simple directed paths} for which every element in the step set $\stepset$ is of the form $(1,b)$. 
In other words, these walks constantly move one step to the right. We introduce the abbreviation $\stepset = \{ b_1, \ldots, b_m \}$ in this case.

	Along these restrictions, we introduce the following classes (see Table \ref{tab:dirTypes}):
		A {\it bridge}\index{lattice path!bridge} is a path whose end-point $\walk{n}$ lies on the $x$-axis.
		A {\it meander}\index{lattice path!meander} is a path that lies in the quarter plane $\Z_+^2$.
		An {\it excursion}\index{lattice path!excursion} is a path that is at the same time a meander and a bridge. 
		Their generating functions have been fully characterized in \cite{bafl02} by means of analytic combinatorics, see \cite{flaj09}. 

\begin{table}[ht]
	\begin{center}
	\begin{tabular}{|c|c|c|}
		\hline                        & ending anywhere & ending at $0$ \\
		\hline 
				                              & \multirow{3}{*}{\includegraphics[width=5cm]{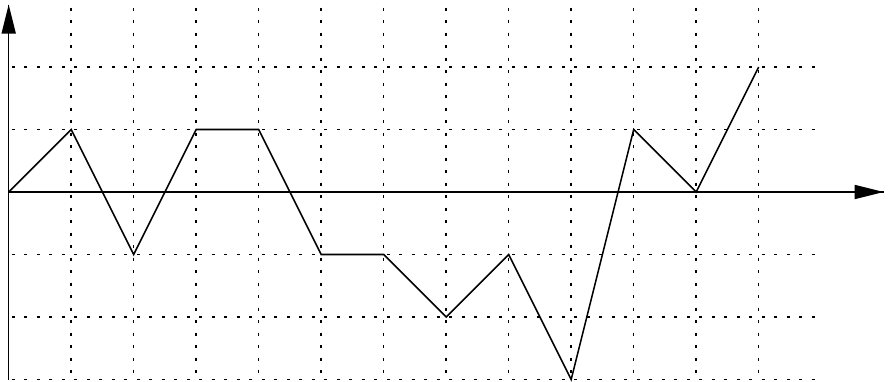}} & \multirow{3}{*}{\includegraphics[width=5cm]{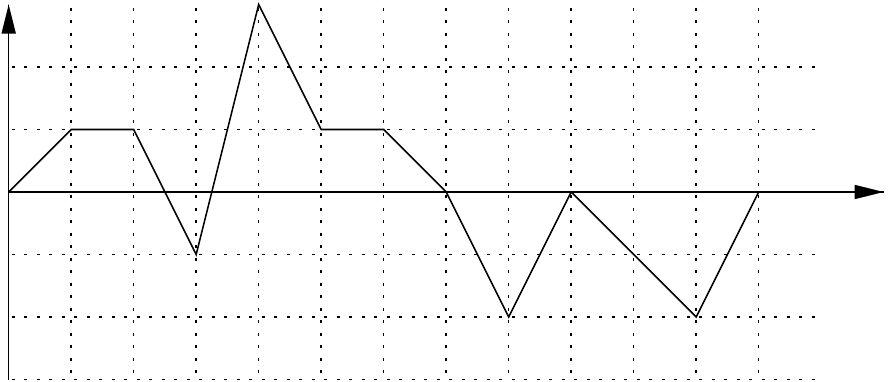}}\\
				                              & & \\
				                              & & \\
				                unconstrained	& & \\
				                	(on $\Z$)	  & & \\
		                                  &  walk/path ($\Wc$)                   & bridge ($\Bc$)\\
		                                  &  $W(z) = \frac{1}{1-zP(1)}$          & $B(z) = z \frac{u_1'(z)}{u_1(z)}$\\ 
	\hline 
				                   & \multirow{3}{*}{\includegraphics[width=5cm]{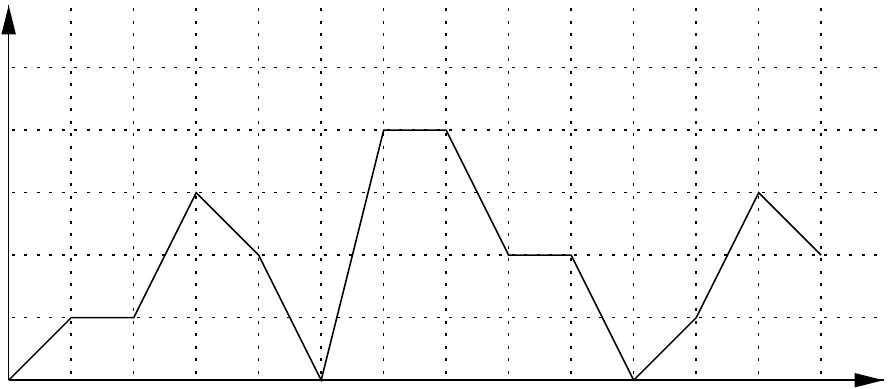}} & \multirow{3}{*}{\includegraphics[width=5cm]{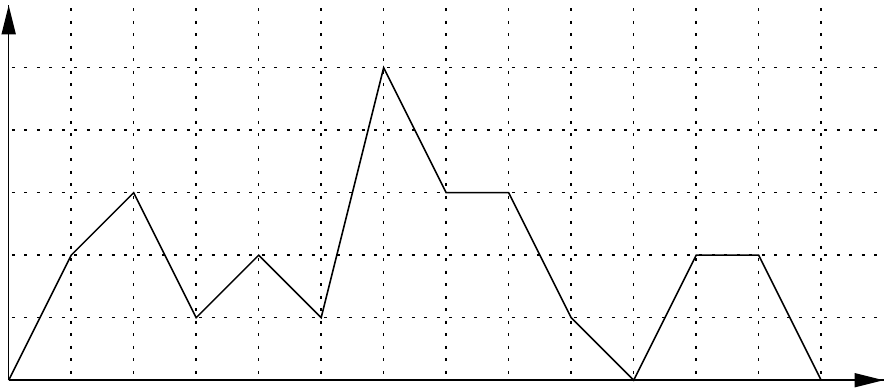}}\\
				                   & & \\
				                   & & \\
				                constrained	& & \\
				                (on $\Z_+$)	& & \\
		                                &  meander ($\Mc$)                       & excursion ($\Ec$)\\
		                                &  $M(z) = \frac{1-u_1(z)}{1-zP(1)}$          & $E(z) = \frac{u_1(z)}{p_{-1}z}$\\
		\hline
\end{tabular}
\end{center}
\caption{The four types of paths: walks, bridges, meanders and excursions, and the corresponding generating functions for Motzkin paths {\cite[Fig. 1]{bafl02}}.}
\label{tab:dirTypes}
\end{table}

\begin{definition}[Motzkin paths]
	\label{ex:motzkin}
	A Motzkin path is a path that starts at the origin and is given by the step set $\stepset = \{-1,0,+1\}$. 
\end{definition}

	We will refer to Motzkin walks/meanders/bridges/excursions depending on the different restrictions. 
	In the literature Motzkin paths are often defined as Motzkin excursions, e.g.~in \cite{dosh77}.
			
	In many situations it is useful to associate weights to single steps. 
	
	\begin{definition}[Weights]
		For a given step set $\stepset$, we define the respective {\it system of weights} 
	as $ \{p_s~|~s \in \stepset\}$ where $p_s >0$ is the associated weight to step $s \in \stepset$. 
		The {\it weight of a path} is defined as the product of the weights of its individual steps. 
	\end{definition}

	A typical weighted lattice path model is $p_s=1$ (enumeration of paths), or $\sum_s p_s = 1$ (probabilistic model of paths, i.e.,~step $s$ is chosen with probability $p_s$).
	
	The following definition is the algebraic link between weights and steps. It is given only for the case of Motzkin paths, which is sufficient for our purpose. 
			
	\begin{definition}[Jump polynomial of Motzkin paths]
		The \emph{jump polynomial} is defined as the polynomial in $u, u^{-1}$ (a Laurent polynomial)
		\begin{align*}
			P(u) := p_{-1}u^{-1} + p_{0} + p_1 u, \qquad \text{ with } \qquad  p_{-1},p_0,p_1 > 0.
		\end{align*}
		The \emph{kernel equation} is defined by
		\begin{align*}
			1-zP(u) &= 0, &&\text{ or equivalently } &
			u - z(u P(u)) &=0.
		\end{align*}
		The quantity $K(z,u) := u - zu P(u)$ is called \emph{kernel}.
	\end{definition}
	
	A walk is called \emph{periodic} with period $p$ if there exists a polynomial $H(u)$ and integers $b \in \Z$ and $p \in \N$, $p>1$ such that $P(u) = u^{b} H(u^p)$. Otherwise its called \emph{aperiodic}. 
	The condition $p_0 >0$ implies aperiodicity for Motzkin paths.
	Note that generating functions of aperiodic walks possess a unique singularity on the positive real axis \cite{bafl02}.  
	
	The kernel plays a crucial \role and is name-giving for the \emph{kernel method}, which is the key tool characterizing this family of lattice paths. The interested reader is referred to \cite[Chapter 2]{bafl02}. 
	In the heart of this method lies the observation that the kernel equation is of degree $2$ in $u$, and therefore has generically $2$ roots. These correspond to branches of an algebraic curve given by the kernel equation. 
	
	\begin{prop}[Roots of the kernel]
		\label{prop:rootsofthekernel}
		The kernel equation $1-zP(u)=0$ has $2$ solutions:
		\begin{align*}
			u_{1,2}(z) &= \frac{1-p_0 z \mp \sqrt{(1-p_0 z)^2 - 4 p_{-1}p_{1} z^2}}{2 p_{1} z}.
		\end{align*}
	\end{prop}	
	
	It holds that $\lim_{z \to 0} u_1(z) = 0$, and $\lim_{z \to 0} u_2(z) = \infty$.
	Because of that, we call $u_1(z)$ the small branch, and $u_2(z)$ the large branch.
	
	Banderier and Flajolet showed that the generating functions of bridges, excursions and meanders can be expressed in terms of the small branch(es) and the jump polynomial, see Table~\ref{tab:dirTypes}. 	
	The branch $u_1(z)$ is real positive near $0$. 
	It is responsible for the asymptotic behavior of bridges, excursions and meanders, compare \cite[Theorem 3 and 4]{bafl02}. 
	
	In order to understand their behavior we need the following constants:
	
	\begin{lemma}[Structural constants]
		\label{lem:signMotzkinTauPtau}
		The \emph{structural constant} $\tau$, which is the unique positive solution of $P'(u)=0$, is $\tau = \sqrt{\frac{p_{-1}}{p_1}}$. The \emph{structural radius} is $\rho = \frac{1}{P(\tau)} = \frac{1}{p_0 + 2 \sqrt{p_{-1}p_1}}$.
	\end{lemma}
	
	The theory of Newton-Puiseux series implies that the small branch $u_1(z)$ is analytic on the open interval $(0,\rho)$, and satisfies the singular expansion
	\begin{align}
		\label{eq:u1asy}
		u_1(z) &= \tau - C \sqrt{1-\frac{z}{\rho}} + \LandauO\left(1 - \frac{z}{\rho} \right),
	\end{align}
	for $z \to \rho^-$, where $C = \sqrt{2 \frac{P(\tau)}{P''(\tau)}}$. 
	This is a direct consequence of the implicit function theorem. 

\begin{prop}[Square-root singularity]
	\label{prop:decompu1}
	There exists a neighborhood $\neighrayrho$ such that for $z \to \rho$ in $\neighrayrho$ $u_1(z)$ has a local representation of the kind
	\begin{align*}
		u_1(z) &= a(z) + b(z) \sqrt{1-z/\rho}, \quad \text{ with } \quad 
		a(\rho)=\tau, \text{ and } b(\rho) = -C,
	\end{align*}
	where $a(z)$ and $b(z)$ are analytic functions for every point $z \in \neighrayrho$, $z \neq z_0$. 
\end{prop}

\begin{proof}
	This is a direct consequence of the explicit structure of $u_1(z)$ from Proposition~\ref{prop:rootsofthekernel}.
\end{proof}

\section{Properties of Motzkin paths}
\label{sec:motzprop}

The following examples are motivated by the very nice presentation of Feller \cite[Chapter~III]{fell68} about one-dimensional symmetric, simple random walks. Therein, the discrete time stochastic process~$(S_n)_{n \geq 0}$ is defined by $S_0 = 0$ and $S_n = \sum_{j=1}^n X_j$, $n \geq 1$, where the $(X_i)_{i\geq 1}$ are iid Bernoulli random variables with $\PR[X_i = 1] = \PR[X_i = -1] = \frac{1}{2}$. 
These results are generalized to the case of Motzkin paths. In particular compare \cite[Problems~$9$-$10$]{fell68} and \cite[Remark of Barton]{SkSh57} for returns to zero of symmetric and asymmetric random walks, respectively. Furthermore, see \cite[Chapter~III.$5$]{fell68} for sign changes, and \cite[Chapter~III.$7$]{fell68} for the height. 
See also the recent paper of D\"obler \cite{dobl15} on Stein's method for this questions in which he derives bounds for the convergence rate in the Kolmogorov and the Wasserstein metric. 

Let us now analyze these properties in the case of Motzkin walks. 
For the sake of brevity we will only mention the weak convergence law. However, in all cases the local law and the asymptotic expansions for mean and variance hold as well.

\subsection{Returns to zero}

A \emph{return to zero} is a point of a walk of altitude $0$, except for the starting point; in other words a return to the $x$-axis, see Figure~\ref{fig:signchanges}. 
In order to count them we consider ``minimal'' bridges, in the sense that the bridges touch the $x$-axis only at the beginning and at the end. 
We call them \emph{arches}.
As a bridge is a sequence of such arches, we get their generating function in the form of
$
	A(z) = 1 - \frac{1}{B(z)}.
$

\begin{lemma}
	\label{lem:archdecomp}
	The generating function of arches $A(z)$ is for $z \to \rho$ of the kind
	\begin{align*}
		A(z) &= a(z) + b(z) \sqrt{1-z/\rho},
	\end{align*}
	where $a(z)$ and $b(z)$ are analytic functions in a neighborhood $\neighrayrho$ of $\rho$ (i.e.,~for $z \in \neighrayrho$ it holds that $z \notin (\rho,\infty)$).
\end{lemma}

\begin{proof}
	We know that $B(z) = z \frac{u_1'(z)}{u_1(z)}$ is analytic for $|z| < \rho$, see \cite[Theorem 3]{bafl02}. Due to $p_0>0$ (aperiodicity) $\rho$ is the only singular point on the circle of convergence. Hence, 
	\begin{align}
		\label{eq:signBtau1}
		B(z) &= \frac{C_1}{\sqrt{1-z/\rho}} + \LandauO(1), \qquad C_1 = \frac{C}{2 \tau},
	\end{align} 
	by \eqref{eq:u1asy} for $z \to \rho$. Proposition \ref{prop:decompu1} together with \eqref{eq:signBtau1} implies the desired decomposition.
\end{proof}

Here, we are interested in the number of returns to zero of walks which are unconstrained by definition. Every walk can be decomposed into a maximal initial bridge, and a walk that never returns to the $x$-axis%
, see Figure~\ref{fig:retzerotail}
. Let us denote the generating function of this \emph{tail} by $T(z)$. 

\begin{figure}[ht]
	\begin{center}	
		\includegraphics[width=0.7\textwidth]{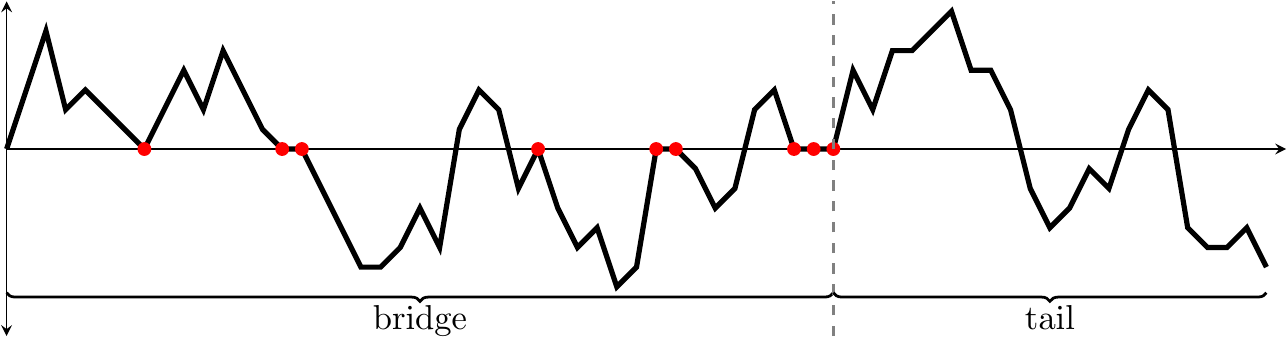}
		\caption{A walk with $9$ returns to zero decomposed into a bridge and a tail.}
		\label{fig:retzerotail}
	\end{center}
\end{figure}

As we want to count the number of returns to zero, we mark each arch by an additional parameter $u$ and reconstruct the generating function of walks. This gives 
\begin{align*}
	W(z,u) &= \frac{1}{1 - u A(z)} T(z) = \frac{W(z)}{u + (1-u)B(z)}, && \text{ with } &
	T(z) &= \frac{W(z)}{B(z)}.
\end{align*}

Let us define the random variable $X_n$ which stands for the number of returns to zero of a random meander of length $n$. Thus,
$
	\PR[X_n = k] = \frac{[u^k z^n] W(z,u)}{[z^n] W(z,1)}. 
$

\begin{theo}[Limit law for returns to zero]
	\label{theo:mainRetZero}
	Let $X_n$ denote the number of returns to zero of a walk of length $n$. Let $\delta=P'(1)$ be the drift. 
	\begin{enumerate} 
		\item For $\delta \neq 0$ we get convergence to a geometric distribution:
			\begin{align*}
				X_n \stackrel{d}{\to} \operatorname{Geom}\left(\frac{|p_1 - p_{-1}|}{P(1)}\right).
			\end{align*}
		\item For $\delta = 0$ we get convergence to a half-normal distribution:
			\begin{align*}
				\frac{X_n}{\sqrt{n}} \stackrel{d}{\to} \Hc\left(\sqrt{\frac{P(1)}{P''(1)}}\right).
			\end{align*}
	\end{enumerate}
\end{theo}

\begin{proof}
	First of all, we see that $[z^n] W(z,1) = [z^n]W(z) = P(1)^n$. 
	Note that because of $p_0>0$ (aperiodicity) $B(z)$ is singular only at $\rho$. Obviously, $W(z)$ is singular at $\rho_1 := \frac{1}{P(1)}$. 
	
	Note that $P(\tau)$ is the unique minimum of $P(u)$ on the positive real axis. Hence, only two cases are possible: $\rho_1 < \rho$, if $\tau \neq 1$; or $\rho_1 = \rho$, if $\tau = 1$. These cases are equivalent to $\delta \neq 0$ and $\delta = 0$, respectively. In the first case $W(z)$ is responsible for the dominant singularity. Then we get ($B(z)$ is analytic for $|z|<\rho$)
	\begin{align*}
		[z^n] W(z,u) &= \frac{1}{B\left(\rho_1\right)} \frac{P(1)^n}{1 - u\left(1-\frac{1}{B\left(\rho_1\right)}\right)} + \Landauo(P(1)^n).
	\end{align*}
	Thus, the limit distribution is a geometric distribution with parameter $\lambda = \frac{1}{B\left(\rho_1\right)}$. Distinguishing between a positive and a negative drift, and some tedious calculations with the help of relations implied by the kernel equation, give the final result for $\delta \neq 0$. 

	In the second case $\tau=1$ or $\delta=0$, we apply Theorem \ref{theo:theo4}. By Lemma \ref{lem:archdecomp} it holds that $1/W(z,u)$ has a decomposition of the kind \eqref{eq:decompFinv}.
	In particular, from \eqref{eq:signBtau1} we directly get that
	\begin{align*}
		\frac{1}{W(z,u)} &= \left(1-\frac{z}{\rho} \right)u + \frac{C}{2}(1-u) \sqrt{1-\frac{z}{\rho}} + \LandauO\left( \left(1-\frac{z}{\rho} \right) (1-u) \right), 
	\end{align*}
	for $z \to \rho$ and $u \to 1$, with $g(\rho,1)=h(\rho,1)=g_u(\rho,1)=g_{uu}(\rho,1)=0$; and $g_z(\rho,1)=-P(1)$ and $h_u(\rho,1) = - \sqrt{\frac{P(1)}{2 P''(1)}}$. Hence, Theorem \ref{theo:theo4} yields the result.
\end{proof}

\subsection{Sign changes of Motzkin walks}

We say that nodes which are strictly above the $x$-axis have a \emph{positive sign} denoted by ``$+$'', whereas nodes which are strictly below the $x$-axis have a \emph{negative sign} denoted by ``$-$'', and nodes on the $x$-axis are \emph{neutral} denoted by ``$0$''. 
This notion easily transforms a walk $\omega = (\omega_n)_{n \geq 0}$ into a sequence of signs. In such a sequence a \emph{sign change} is defined by either the pattern~$+(0)-$ or~$-(0)+$, where~$(0)$ denotes a non-empty sequence of $0$'s, see Figure \ref{fig:signchanges}.

\begin{figure}[ht]
	\begin{center}	
		\includegraphics[width=0.8\textwidth]{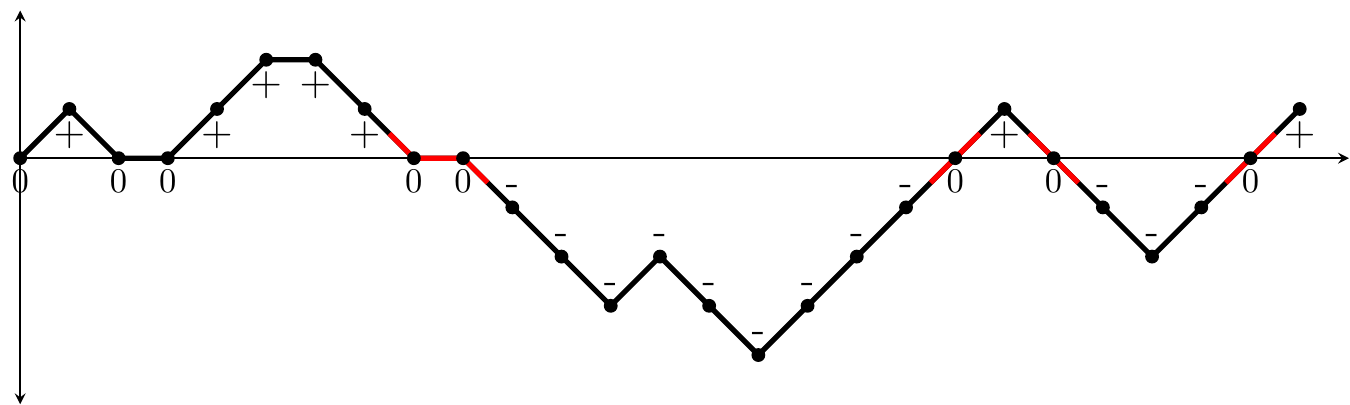}
		\caption{A Motzkin walk with $7$ returns to zero and $4$ sign changes. The positive, neutral or negative signs of the walks are indicated by $+,0$, or $-$, respectively.}
		\label{fig:signchanges}
	\end{center}
\end{figure}

The main observation in this context is the non-emptiness of the sequence of $0$'s.
Geometrically this means that it has to touch the $x$-axis when passing through it. This means that we can count the number of sign changes by counting the number of maximal parts above or below the $x$-axis. 
The idea is to decompose a walk into an alternating sequence of positive (above the $x$-axis) and negative (below) excursions terminated by a positive or negative meander. 

We introduce two new terms: \emph{positive excursions} are ``traditional'' excursions, i.e.,~they are required to stay above the $x$-axis, whereas \emph{negative excursions} are walks which start at zero, end on the $x$-axis, but are required to stay below the $x$-axis. 

\begin{lemma}
	\label{lem:posnegexc}
	Among all walks of length $n$, the number of positive excursions is equal to the number of negative excursions. 
\end{lemma}

\begin{proof}
	Mirroring bijectively maps positive excursions to negative ones. 
\end{proof}

We define the bivariate generating function 
$
	B(z,u) = b_{n,k}z^n u^k,
$
where $b_{n,k}$ denotes the number of bridges of size $n$ having $k$ sign changes. Furthermore, we define 
\begin{align*}
	C(z) = \frac{1}{1-p_0 z},
\end{align*}
as the generating function of \emph{chains}, which are walks constructed solely from the jumps of height~$0$. Then the generating function of excursions starting with a $+1$ jump is given by
\begin{align*}
	E_1(z) &= \frac{E(z)}{C(z)}-1,
\end{align*}
because we need to exclude all excursions which start with a chain or are a chain. 
Due to Lemma~\ref{lem:posnegexc} this is also the generating function for excursions starting with a $-1$ jump.

\begin{theo}
	\label{theo:signbridgedecomp}
	The bivariate generating function of bridges 
	(where $z$ marks the length, and $u$ marks the number of sign changes of the walk) 
	is given by
	\begin{align*}
		B(z,u) &= C(z) \left( 1 + \frac{2 E_1(z)}{1-uE_1(z)} \right).
	\end{align*}
\end{theo}

\begin{proof}
	A bridge is either a chain, which has zero sign changes, or it is not a chain. In the latter it is an alternating sequence of positive and negative excursions, starting with either of them. We decompose it uniquely into such excursions, by requiring that all except the first one start with a non-zero jump. Therefore the first excursion is counted by $E(z)-C(z)$, whereas all others are counted by $E_1(z)$. The decomposition is shown in Figure~\ref{fig:signdecomp}.
\end{proof}

\begin{figure}[ht]
	\begin{center}	
		\includegraphics[width=0.9\textwidth]{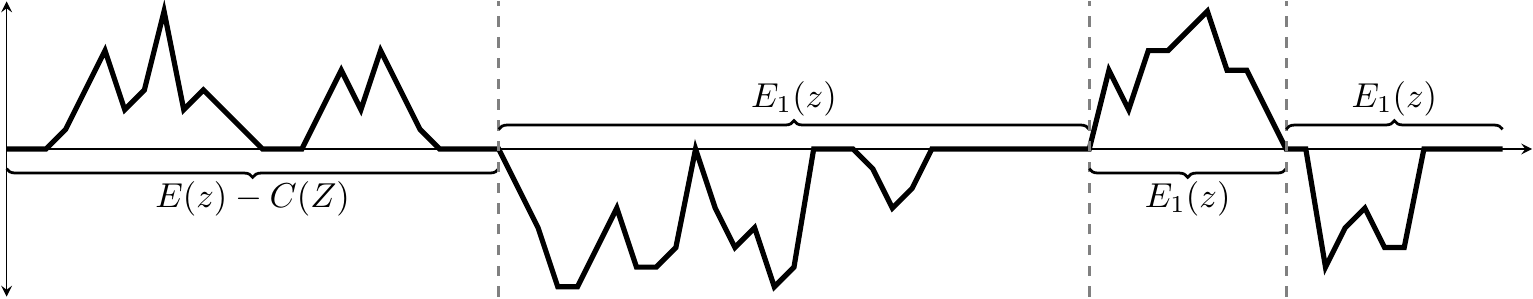}
		\caption{A bridge is an alternating sequence of positive and negative excursions. Here, it 
		starts with a positive excursion, followed by excursions starting with a non-zero jump.}
		\label{fig:signdecomp}
	\end{center}
\end{figure}

Let $X_n$ be the random variable for the number of sign changes of a random bridge of length $n$. 
Thus,
$
	\PR[X_n = k] = \frac{[u^k z^n] B(z,u)}{[z^n] B(z,1)}. 
$

\begin{theo}[Limit law for returns to zero for bridges]
	\label{theo:signbridges}
	Let $X_n$ denote the number of sign changes of a Motzkin bridge of length $n$. Then for $n \to \infty$ the normalized random variable has a Rayleigh%
		\footnote{The parameter $\lambda = \sigma^{-2}$ was used in \cite[Theorem 1]{DrSo97}. 
			}
		limit distribution
	\begin{align*}
		\frac{X_n}{\sqrt{n}} \stackrel{d}{\to} \Rc\left(\sigma\right) \qquad \text{ and } \qquad
		\sigma = \frac{\tau}{2} \sqrt{\frac{P''(\tau)}{P(\tau)}},
	\end{align*}
	where $\tau = \sqrt{\frac{p_{-1}}{p_1}}$ and $\Rc(\sigma)$ has the density $\frac{x}{\sigma^2} \exp\left(-\frac{x^2}{2\sigma^2}\right)$ for $x \geq 0$. 
\end{theo}

\begin{proof}[(Sketch)]
	We apply the first limit theorem of Drmota and Soria, \cite[Theorem 1]{DrSo97}. 
	Proposition~\ref{prop:decompu1} implies that $E_1(z)$ and therefore $B(z,u)$ has a decomposition of the desired form~\eqref{eq:decompFinv}.
	Checking the other conditions with the help of Lemma~\ref{lem:signMotzkinTauPtau} yields the result.
\end{proof}

Finally, we consider sign changes of walks. 
Since we want to count the number of sign changes we need to know whether a bridge ended with a positive or negative sign. 
Let \emph{positive bridges} be bridges whose last non-zero signed node was positive, and \emph{negative bridges} be bridges whose last non-zero signed node was negative. Their generating functions are denoted by $B_+(z,u)$ and $B_-(z,u)$, respectively. Figure~\ref{fig:signdecomp} shows a negative bridge.

\begin{lemma}
	\label{lem:B+}
	The number of positive and negative bridges is the same and given by
	\begin{align*}
		B_+(z,u) = \frac{B(z,u)-C(z)}{2} = \frac{E(z)-C(z)}{1-uE_1(z)}.
	\end{align*}
\end{lemma}

\begin{proof}
	The result is a direct consequence of Lemma \ref{lem:posnegexc}, because a positive bridge is either a non trivial excursion or a negative bridge where an additional excursion starting with a $+1$ jump was appended. For negative bridges an analogous construction holds. 
\end{proof}

\begin{prop}
	The bivariate generating function of walks $W(z,u) = \sum_{n,k \geq 0} w_{nk} z^n u^k$ where $w_{nk}$ is the number of all walks of length $n$ with $k$ sign changes, is given by
	\begin{align*}
		W(z,u) &= B(z,u) \frac{W(z)}{B(z)} + B_+(z,u) \left(\frac{W(z)}{B(z)} - 1 \right) (u-1),
	\end{align*}
	where $W(z) = \frac{1}{1-zP(1)}$ is the generating function of walks.
\end{prop}

\begin{proof}
	Combinatorially, a walk is either a bridge or a bridge concatenated with a meander that does not return to the $x$-axis again. In the second case an additional sign change appears if the bridge ends with a positive sign and continues with a meander always staying above the $x$-axis, or vice versa. By Lemma~\ref{lem:B+} the desired form follows.
\end{proof}

The next theorem concludes this discussion. Its proof is similar to the one of Theorem~\ref{theo:mainRetZero}. 

\begin{theo}[Limit law for sign changes]
	\label{theo:mainSignChanges}
	Let $X_n$ denote the number of sign changes of Motzkin walks of length $n$. Let $\delta=P'(1)$ be the drift. 
	\begin{enumerate} 
		\item For $\delta \neq 0$ we get convergence to a geometric distribution:
			\begin{align*}
				X_n \stackrel{d}{\to} \operatorname{Geom}\left( \lambda \right), &&& 
				\text{ with } & 
				\lambda &= 
					\begin{cases}
						\frac{p_1}{p_{-1}}, & \text{ for } \delta < 0,\\
						\frac{p_{-1}}{p_{1}}, & \text{ for } \delta > 0.
					\end{cases}
			\end{align*}
		\item For $\delta = 0$ we get convergence to a half-normal distribution:
			\begin{align*}
				\frac{X_n}{\sqrt{n}} \stackrel{d}{\to} \Hc\left(\frac{1}{2} \sqrt{\frac{P''(1)}{P(1)}}\right).
			\end{align*}
	\end{enumerate}
\end{theo}

\subsection{Height of Motzkin walks}

For a path of length $n$ we define the \emph{height} as its maximally attained $y$-coordinate, see Figure~\ref{fig:motzkinheight}. Formally, let $\omega = (\omega_k)_{k=0}^n$ be a walk. 
Then its height is given by
$
	\max_{k \in \{0,\ldots,n\}} \omega_k.
$

\begin{figure}[ht]
	\begin{center}	
		\includegraphics[width=0.8\textwidth]{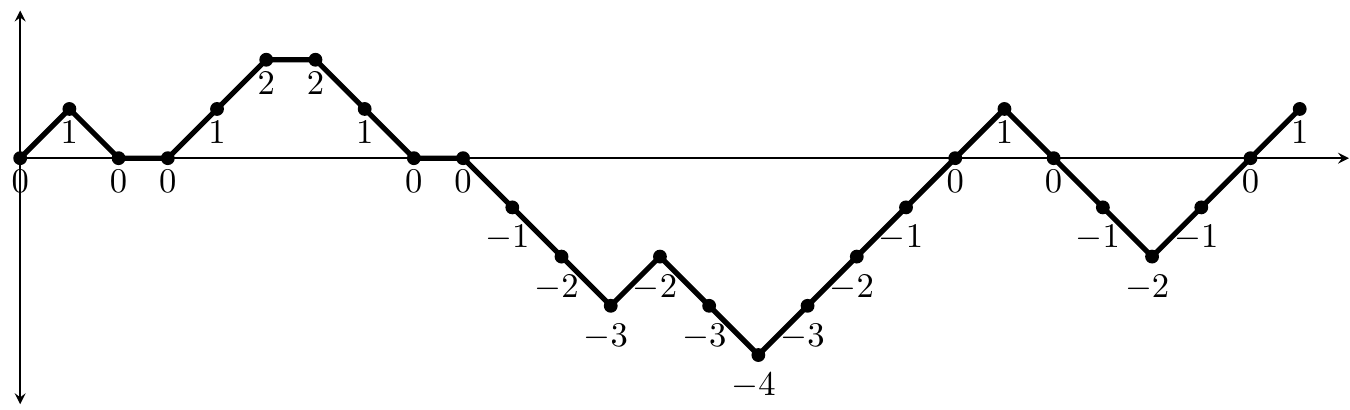}
		\caption{A Motzkin walk of height $2$. The relative heights are given at every node.}
		\label{fig:motzkinheight}
	\end{center}
\end{figure}

In order to analyze the distribution of heights, we define the bivariate generating function $F(z,u) = \sum_{n,h \geq 0} f_{nh}z^n u^h$. The coefficient $f_{nh}$ represents the number of walks of height $h$ among walks of length $n$. 
First we need a relation between the branches of the kernel equation:

\begin{lemma}
	\label{lem:u1v1relation}
	Let $P(u) = p_{-1}u^{-1} + p_0 + p_1 u$. Then the small branch $u_1(z)$ and the large branch $u_2(z)$ of the kernel equation $1-zP(u)=0$ fulfil
	\begin{align*}
		u_1(z) u_2(z) &= \frac{p_{-1}}{p_1} & \text{ and } && u_1(z) + u_2(z) &= \frac{1-zp_0}{zp_{1}}.
	\end{align*}
\end{lemma}

\begin{proof}
	The kernel equation factorizes into $u(1-zP(u)) = -zp_1 (u-u_1(z)) (u-u_2(z))$. Comparing the coefficients gives the results.
\end{proof}

This relation gives us an explicit expression of $F(z,u)$ in terms of the large and small branch. For the final analysis we will use the latter.

\begin{theo}
	The bivariate generating function of Motzkin walks 
	(where $z$ marks the length, and $u$ marks the height of the walk) 
	is given by
	\begin{align*}
		F(z,u) &= \frac{1}{1-zP(1)} \frac{u_2(z)-1}{u_2(z)-u}
		        = \frac{1}{1-zP(1)} \frac{1-\frac{p_{1}}{p_{-1}}u_1(z)}{1-u\frac{p_{1}}{p_{-1}}u_1(z)}.
	\end{align*}
\end{theo}

\begin{proof}
	Banderier and Nicod{\`e}me derived in \cite[Theorem 2]{bani10} the generating function $F^{[-\infty, h]}(z)$ for walks staying always below a wall $y=h$. For the case of Motzkin walks we get
	$
		F^{[-\infty, h]}(z) = \frac{1 - \left(\frac{1}{u_2(z)}\right)^{h+1}}{1-zP(1)},
	$
	where $u_2(z)$ is the large branch of the kernel equation. From this we directly get the generating function $F^{[h]}(z)$ for walks that have height exactly $h$. For $h \geq 1$ it equals
	\begin{align*}
		F^{[h]}(z) &= F^{[-\infty, h]}(z) - F^{[-\infty, h-1]}(z) = \frac{u_2(z)-1}{1-zP(1)} \left(\frac{1}{u_2(z)}\right)^{h}.
	\end{align*}
	The last formula also holds for $h=0$. Finally, marking the heights by $u$ and summing over all possibilities yields the result. 
	The second formula is a consequence of Lemma~\ref{lem:u1v1relation}.
\end{proof}

Let $X_n$ be the random variable for the height of a random walk of length $n$. Thus,
$
	\PR[X_n = k] = \frac{[u^k z^n] F(z,u)}{[z^n] F(z,1)} = \frac{[u^k z^n] F(z,u)}{P(1)^n}. 
$
This time the behavior will be different for $\delta <0$ and $\delta >0$. 
We omit its proof, however the ideas are again similar to the ones of Theorem~\ref{theo:mainRetZero}. 

\pagebreak
\begin{theo}[Limit law for the height]
	\label{theo:height}
	Let $X_n$ denote the height of a Motzkin walk of length $n$. Let $\delta=P'(1)$ be the drift. 
	\begin{enumerate} 
		\item For $\delta < 0$ we get convergence to a geometric distribution:
			\begin{align*}
				X_n \stackrel{d}{\to} \operatorname{Geom}\left( \frac{p_1}{p_{-1}} \right) .
			\end{align*}
		\item For $\delta = 0$ the standardized random variable converges to a half-normal distribution:
			\begin{align*}
				\frac{X_n}{\sqrt{n}} \stackrel{d}{\to} \Hc\left(\sqrt{\frac{P''(1)}{P(1)}}\right).
			\end{align*}
		\item For $\delta > 0$ the standardized random variable converges to a normal distribution:
			\begin{align*}
				\frac{X_n-\mu n}{\sigma \sqrt{n}} &\stackrel{d}{\to} \Nc\left( 0,1 \right), &
				\mu &= \frac{\delta}{P(1)}, &
				\sigma^2 &= 
					1 - \frac{p_0}{P(1)} - \left(\frac{\delta}{P(1)}\right)^2.
			\end{align*}		
	\end{enumerate}
\end{theo}

\vspace{-2mm}

\section{Conclusion}
\label{sec:conclusion}

Drmota and Soria \cite{DrSo97} presented three schemata leading to three different limiting distributions: Rayleigh, normal, and a convolution of both. 
This paper can be seen as an extension, by adding Theorem~\ref{theo:theo4} yielding a half-normal distribution to this family. 
Other popular limit theorems are Hwang's quasi-powers theorem \cite{hwan98}, and (implied by it) the supercritical composition scheme \cite[Proposition~IX.6]{flaj09}. These lead to a normal distribution.

The question may arise, how Theorem~\ref{theo:theo4} behaves in the situation of a singularity $\rho(u)$ with $\rho'(1) \neq 0$ and $\rho''(1) \neq 0$, compare Remark~\ref{rem:strongermaintheo}. This remains an object for future research. 

However, the more interesting question is if more ``natural'' appearances of such situations exist. 
Another known example is the limit law of the final altitude of meanders with zero drift in the reflection-absorption model in \cite{bawa15c}. 
Chronologically, this was the starting point for the research of this paper.
But this distribution also appears in number theory, see \cite{ggiv07}.

Yet another question is how the zero drift behavior of the analyzed parameters generalizes to other lattice path models. 
We will comment on these questions in the full version of this work.

	Summing up, the applications to Motzkin paths show that intuition might lead you into the wrong direction. 
	In Table~\ref{tab:compretsign} we see a comparison of the parameters. Obviously, the situation depends strongly on the drift. The critical case of a $0$ drift seems to be the most delicate one, as it changes the nature of the law. In this case the limiting probability functions are concentrated at $0$. In particular the expected value for $\Theta(n)$ trials grows like $\Theta(\sqrt{n})$ and not linearly. Equipped with the presented tools they might still be a ``shock to intuition and common sense'' but should not come  ``unexpected'' anymore.

	\begin{table}[ht]
	\begin{center}
	\begin{tabular}{|c||c|c|c|}
		\hline drift      & returns to zero & sign changes & height \\
		\hline\hline $\delta<0$ & 
			$\operatorname{Geom}\left( \frac{p_{-1}-p_{1}}{P(1)} \right)$ & 
			$\operatorname{Geom}\left( \frac{p_{1}}{p_{-1}} \right) $ & 
			$\operatorname{Geom}\left( \frac{p_{1}}{p_{-1}} \right) $
			\\
		\hline $\delta=0$ & 
			$\Hc\left(\sqrt{\frac{P(1)}{P''(1)}}\right)$ & 
			$\Hc\left(\frac{1}{2} \sqrt{\frac{P''(1)}{P(1)}}\right) $ & 
			$\Hc\left(\sqrt{\frac{P''(1)}{P(1)}}\right) $
			\\
		\hline $\delta>0$ & 
			$\operatorname{Geom}\left( \frac{p_{1}-p_{-1}}{P(1)} \right)$ & 
			$\operatorname{Geom}\left( \frac{p_{-1}}{p_{1}} \right)$ & 
			Normal distribution
			\\
		\hline
  \end{tabular}
\end{center}
\caption{Summary of the limit laws for Motzkin paths.}
\label{tab:compretsign}
\end{table}

\vspace{-6mm}

\phantomsection
\addcontentsline{toc}{section}{Acknowledgments}
\section*{Acknowledgments}
\label{sec:ack}
The author is grateful to Cyril Banderier, Bernhard Gittenberger, and Michael Drmota for insightful discussions. 
The author also thanks the anonymous referees for their suggested improvements.
This work was supported by the Austrian Science Fund (FWF) grant SFB F50-03. 
\phantomsection

\clearpage
\phantomsection
\addcontentsline{toc}{section}{Bibliography} 
\bibliographystyle{plain}
\bibliography{literature_LP}

\end{document}